\documentclass[11pt,]{amsart}

\usepackage{libertine}
\usepackage{euler}
\usepackage[T1]{fontenc}

\usepackage{verbatim}
\usepackage{amsmath}
\usepackage{amssymb}
\usepackage{amsfonts}
\usepackage{amsthm}
\usepackage{mathrsfs}
\usepackage{tikz}
\usetikzlibrary{matrix,arrows}
\usepackage[all]{xy}
\usepackage{mathrsfs}
\usepackage{a4wide}
\usepackage{enumerate}
\usepackage{hyperref}
\usepackage{parskip}

\theoremstyle{plain}

\newtheorem{theorem}{Theorem}
\newtheorem{lemma}[theorem]{Lemma}
\newtheorem{proposition}[theorem]{Proposition}
\newtheorem{corollary}[theorem]{Corollary}
\newtheorem{claim}[theorem]{Claim}

\theoremstyle{definition}
\newtheorem{definition}[theorem]{Definition}
\newtheorem{example}[theorem]{Example}

\theoremstyle{remark}
\newtheorem{remark}[theorem]{Remark}

\newcommand{\G}{\mathcal{G}}
\newcommand{\E}{\widehat{E}}
\newcommand{\ucong}{\rotatebox{90}{$\cong$}}
\newcommand{\X}{\widehat{X}}
\title{On the Baum-Connes conjecture for Gromov monster groups} 
\date{September 2014}
\author{Martin Finn-Sell}
\thanks{Martin Finn-Sell: Georg-August Universit\"{a}t G\"{o}ttingen, email: mfinnse@uni-math.gwdg.de}
\begin{document}
\bibliographystyle{alpha}
\begin{abstract}
We present a geometric approach to the Baum-Connes conjecture with coefficients for Gromov monster groups via a theorem of Khoskham and Skandalis. Secondly, we use recent results concerning the a-T-menability at infinity of large girth expanders to exhibit a family of coefficients for a Gromov monster group for which the Baum-Connes conjecture is an isomorphism.
\end{abstract}

\maketitle

\section{Introduction}

It is known that any group $\Gamma$ that contains a coarsely embedded large girth expander sequence $X$ does not have Yu's property A and admits coefficients for which the Baum-Connes conjecture fails to be a surjection, but is an injection\cite{MR1911663,explg1}.

We explore this situation from the point of view of the partial geometry that can be associated to the expander graph $X$ that it inherits from the group $\Gamma$. The advantage of this geometric approach over the analytic one presented in \cite{MR1911663,explg1} is that we can refine it to get positive results to this conjecture via a groupoid construction similar to that of \cite{MR1905840}.

This groupoid approach to coarse geometry first formulates the coarse Baum-Connes conjecture for any uniformly discrete bounded geometry space $X$ using a version of the Baum-Connes conjecture with coefficients for a \'etale groupoid $G(X)$ associated to $X$. This groupoid admits always transformation groupoid decomposition; using Lemma 3.3b) from \cite{MR1905840} we can decompose $G(X)$ using the generators for the metric coarse structure on $X$. Each set of such generators carries its own partial geometric data and these can be studied independently. A systematic approach to this was outlined in \cite{MR2363428}.

When $X$ is coarsely embedded into a group, the group gives rise to a natural generating set of the metric coarse structure on $X$. We outline the construction of these generators below, and one object of this paper is to describe precisely how these generators behave combinatorially. 

\begin{example}\label{ex:1}
Let $\Gamma$ be a finitely generated discrete group with a left invariant word metric and let $f:X \rightarrow \Gamma$ be an injective coarse embedding. We can now identify freely $X$ as a subset of the Cayley graph of $\Gamma$ and so $X$ inherits a metric that shares a coarse type with the original metric. What we gain by doing this is access of the right action of $\Gamma$ on itself. Consider the maps:
\begin{equation*}
t_{g}: \Gamma \rightarrow \Gamma , x \mapsto xg^{-1}
\end{equation*}
We can now restrict these maps to $X$, where they may not be defined everywhere. If we denote the set of points in $X$ with image in $X$ by $D_{g}$ then we have:
\begin{equation*}
t_{g}: D_{g} \rightarrow D_{g^{-1}} , x \mapsto xg^{-1}
\end{equation*}
Let $\mathcal{T}_{X}$ denote the collection of $t_{g}$ restricted to $X$. These maps are partial translations of $X$ (i.e maps that are bijections between subsets of $X$ that move points of $X$ uniformly bounded distances). Additionally, by transitivity properties of the action of $\Gamma$ on itself, these maps generate the metric coarse structure on $X$. This construction is an example of a \textit{partial action} of $\Gamma$.
\end{example}

From these generators we construct a second countable, locally compact, Hausdorff, \'etale groupoid following the techniques of \cite{MR2419901}, which explicitly implements the transformation decomposition of \cite{MR1905840}. The techniques we use here are purely semigroup and order theoretic, following the ideas of \cite{MR2041539,MR2465914,MR2419901,MR2974110}.

A natural problem we have to consider is how to relate this data on $X$ back to $\Gamma$. We solve problem by constructing a Morita equivalence between the groupoid constructed from the data of Example \ref{ex:1} and a transformation groupoid involving an action of $\Gamma$. Results of this nature have been developed that cover the groupoid we construct from Example \ref{ex:1}, i.e those that  that are are constructed from suitable inverse semigroups were notably developed by Khoskham and Skandalis \cite{MR1900993} and further extended by Milan and Steinberg \cite{MR3231226}. Their results all use the following concept:

\begin{definition}\label{def:cocycle}
Let $\G$ be a locally compact groupoid. Then we call a continuous homomorphism from $\G$ to a locally compact group $\Gamma$ a \textit{group valued cocycle}. 
\end{definition}

The main result of Khoshkam and Skandalis \cite{MR1900993} that we will appeal to is stated below. The details concerning cocycles and all the topological criteria are recalled in Section \ref{sect:cocycle}.

\begin{theorem}
Let $\rho: \G \rightarrow \Gamma$ be a continuous, faithful, closed, transverse cocycle. Then there is a universal locally compact Hausdorff $\Gamma$-space $\Omega$ that envelopes the space $\G^{(0)}$ and a Morita equivalence of $\G$ with $\Omega\rtimes \Gamma$.\qed
\end{theorem}

The locally compact Hausdorff space $\Omega$ is called the \textit{Morita envelope} of the cocycle $\rho$, and is constructed as a quotient of the product $\G^{(0)}\times \Gamma$ under an equivalence relation and includes a topologically embedded copy of $\G^{(0)}$.

We will appeal to the theorem above (that we recall as Theorem \ref{Thm:1.8} in the text) to decompose the coarse groupoid $G(X)$ when $X$ is a large girth expander $X$ that is coarsely embedded into a finitely generated group $\Gamma$. This transformation groupoid involving $\Gamma$ will then be Morita equivalent to $G(X)$. We exploit this connection between $G(X)$ and $\Gamma$ to obtain the results concerning the Baum-Connes conjecture with coefficients for $\Gamma$.

In particular we prove two results via a partial translation method, the first of which is:

\begin{theorem}\label{Thm:MTa}
Let $\Gamma$ be a finitely generated group, $X$ be a large girth expander and let $f:X \hookrightarrow \Gamma$ be a coarse embedding. Then the coarse groupoid $G(X)$ is Morita equivalent to $\Omega_{\beta X}\rtimes \Gamma$, where $\Omega_{\beta X}$ is the enveloping locally compact Hausdorff $\Gamma$-space arising from the continuous extension of the partial action of $\Gamma$ on $X$ to $\beta X$.
\end{theorem}

We remark first that $\Omega_{\beta X}$ depends on the unit space of $G(X)$ and second that this result has the results from Section 8 of Willett and Yu \cite{explg1} as a Corollary using homological results of \cite{mypub1}.

The second result requires a \textit{groupoid reduction} of of $G(X)$ to $\partial\beta X$: a groupoid that we call the \textit{coarse boundary groupoid} of $X$ and is denoted by $G(X)|_{\partial\beta X}$. The boundary coarse groupoid $G(X)|_{\partial\beta X}$ satisfies an analogue of Theorem \ref{Thm:MTa}:

\begin{theorem}\label{Thm:MTb}
Let $\Gamma$ be a finitely generated group, $X$ be a large girth expander and let $f:X \hookrightarrow \Gamma$ be a coarse embedding. Then the coarse boundary groupoid $G(X)|_{\partial\beta X}$ is Morita equivalent to $\Omega_{\partial\beta X}\rtimes \Gamma$.
\end{theorem}

By results of \cite{mypub1} this implies the transformation groupoid $\Omega_{\partial\beta X}\rtimes \Gamma$ is a-T-menable, hence we get the following natural corollary:

\begin{corollary}
Let $\Gamma$ be a discrete group that coarsely contains a large girth expander $X$ and let $\Omega_{\partial\beta X}$ be the universal enveloping $\Gamma$-space for the partial action on $\partial\beta X$. Then the Baum-Connes assembly map for $\Gamma$ with coefficients in any $C_{0}(\Omega_{\partial\beta X})$-algebra is an isomorphism.
\end{corollary}
This result is the first positive result concerning the Baum-Connes conjecture with coefficients for these examples of non-exact groups. The author would like to remark that results of a similar nature were obtained in \cite{BGW-exact2013} using purely analytic methods.

\section*{Acknowledgments} The author would like to thank his external thesis examiner, Paul Baum, for being interested in these results and methods. Secondly, the author would like to thank the EPSRC and the University of Southampton through which the authors PhD. was funded.

\numberwithin{theorem}{section}
\setcounter{theorem}{0}

\section{Inverse semigroups and constructing groupoids}
In the following sections we briefly outline the inverse semigroup and groupoid framework that go into the proofs of Theorems \ref{Thm:MTa} and \ref{Thm:MTb}. 

Let $S$ be a semigroup and let $s \in S$. Then $u$ is said to be a \textit{semigroup inverse} for $s$ if $sus=s$ and $usu=u$. A semigroup $S$ is \textit{regular} if every element $s$ has some inverse element $u$, and \textit{inverse} if that inverse element is unique, in which case we denote it by $s^{*}$. Groups are examples of semigroups that are inverse with only a single idempotent. It is clear that in an inverse semigroup $S$ every element $ss^{*}$ is idempotent, and this classifies the structure of idempotents (see, for instance, chapter 5 of \cite{MR1455373}).

We remark also that the idempotents of $S$ form a commutative inverse subsemigroup $E(S)$ that can be partially ordered using this multiplication. Let $e,f \in E(S)$ then:
\begin{equation*}
e \leq f \Leftrightarrow ef=e
\end{equation*}
is a partial order on $E(S)$ that extends to $S$ using:
\begin{equation*}
s \leq t \Leftrightarrow (\exists e \in E) et=s
\end{equation*}

\begin{example}
Let $X$ be a set. The set of all partial bijections of $X$, that is bijections between subsets of $X$, is a inverse monoid that we denote by $I(X)$. If $s$ and $t$ are partial bijections of $X$ then the partial order describes precisely when $s$ is a restriction of $t$ to some subset of the domain of $t$.
\end{example}

In light of this example, we can define the class of inverse semigroups relevent to our study:

\begin{definition}
Let $z \in S$. We say $z$ is a zero element if $z \in E(S)$ and $zs=sz=z$. The empty partial bijection is an example of such a zero element. With this in mind we say $S$ is 0-E-unitary if for every non-zero idempotent $e\in E(S)$ and $s\in S$ we have $e \leq s$ implies $s \in E$.
\end{definition}

For concrete semigroups of partial bijections this condition is equivalent to the restriction of any element never being an idempotent.

Let $S$ be a inverse semigroup. We can define a \text{universal group} for $S$ \cite{MR745358}. This is the group generated by the elements of $S$ with relations $s\cdot t = st$ if $st \not = 0$, and denoted by $U(S)$. Clearly, there is a map $\Phi$ from $S$ to the universal group $U(S)$, given by mapping $s$ to its corresponding symbol in $U(S)$. This map is \textit{not} a homomorphism, but after adjoining a zero element to $U(S)$ does satisfy the inequality $\Phi(st) \leq \Phi(s)\Phi(t)$. Such a map is called a \textit{prehomomorphism}. 

\begin{definition}
Let $S$ be a 0-E-unitary inverse monoid. Then $S$ is \textit{0-F-inverse} if the preimage of each group element in $U(S)$ has a maximal element within the partial order of $S$. In this case we denote these elements by $Max(S)$. The universal group is generated by the set $Max(S)$ with the product: $s\ast t = u$, where $u$ is the unique maximal element above $st \in S$. 

Finally, $S$ is said to be \textit{strongly 0-F-inverse} if it is $0$-F-inverse and there is some group $\Gamma$ and a $0$-restricted idempotent pure homomorphism onto $\Gamma^{0}$; that is a map for which the preimage of $0$ is $0$ and the preimage of the identity contains only idempotents. In particular, this must factor though the universal group, so this amounts to asking if the map $\Phi$ is idempotent pure.
\end{definition}

\subsection{Constructing groupoids}\label{Sect:StoG}
In this section we outline how to build a groupoid given an inverse semigroup $S$ by considering the order structure of $S$ and the action of $S$ on itself. We briefly outline the construction below, but for a full account see \cite{MR2419901}. The main motivation of \cite{MR2419901} is to construct a groupoid that is associated with $S$ that captures the actions of $S$. To obtain it from $S$ itself we require $S$ to act on some space intrisic: namely we use the natural action of $S$ on its idempotent semilattice $E(S)$ by conjugation.

Let $D_{e}=\lbrace f \in E | f \leq e \rbrace$. For $ss^{*} \in E$ we can define a map $\rho_{s}(ss^{*})=s^{*}s$ that extends to $D_{ss^{*}}$ by the formula $\rho_{s}(e) = s^{*}es$. This defines a partial bijection on $E$ from $D_{ss^{*}}$ to $D_{s^{*}s}$. 

We now wish to construct a topological space from $E$ that will admit an action that is induced from the one we have developed above. To do this we consider a subspace of $\textbf{2}^{E}$ given by the functions $\phi$ such that $\phi(0)=0$ and $\phi(ef)=\phi(e)\phi(f)$.

We can topologize this as a subspace of $\textbf{2}^{E}$, where it is a closed subspace (hence compact Hausdorff), with a base of topology given by $\widehat{D}_{e}= \lbrace \phi \in \E | \phi(e)=1 \rbrace$. This space admits a dual action induced from the action of $S$ on $E$ given by the following pointwise equation for every $\phi \in \widehat{D}_{s^{*}s}$:
\begin{equation*}
\widehat{\rho}_{s}(\phi)(e)=\phi(\rho_{s}(e))=\phi(s^{*}es)
\end{equation*}
The use of $\widehat{D}_{e}$ to denote these sets is not a coincidence, as we have a map $D_{e} \rightarrow \widehat{D}_{e}$:
\begin{equation*}
e \mapsto \phi_{e}, \phi_{e}(f)=1 \mbox{ if } e \leq f \mbox{ and } 0 \mbox{ otherwise}.
\end{equation*}

We are now able to construct a groupoid from this data. Consider $\Omega:= \lbrace (s, \phi) | \phi \in D_{s^{*}s} \rbrace$ in the subspace topology induced from $S\times \E$. We then quotient this space by the relation:
\begin{equation*}
(s, \phi) \sim (t, \phi^{'}) \Leftrightarrow \phi=\phi^{'} \mbox{ and } (\exists e \in E) \mbox{ with } \phi \in D_{e} \mbox{ such that } es=et
\end{equation*}
We denote the quotient by $\G_{\E}$. It is possible to turns $\G_{\E}$ into a locally compact groupoid, using the set $\G_{\E}^{(2)}:=\lbrace ([s,x],[t,y]) | t(y)=x \rbrace$, with composition $[s,x][t,y]=[st,y]$, inverse $[s,x]^{-1}=[s^{*},s(x)]$ and slices constructed from the sets $\widehat{D}_{e}$ as a basis of topology. This groupoid is called the \textit{universal groupoid} associated to $S$, and it, in particular, shares the Hilbert space representation theory of $S$ to some extent (this is shown in Corollary 10.16 of \cite{MR2419901} for instance). We would like to remark also that this groupoid can be constructed using an approach based on filters and filter composition \cite{MR3077869}.

\subsection{Group valued cocycles}\label{sect:cocycle}
In this section we make precise the statements about cocycles, partial actions and universal enveloping spaces that we need in the sequel. We follow the references \cite{MR1900993,MR3231226}. 

\begin{definition}
Let $\rho: \G \rightarrow \Gamma$ be a cocycle. We say it is:
\begin{enumerate}
\item \textit{transverse} if the map $\Gamma \times \G \rightarrow \Gamma \times \G^{(0)}$, $(g, \gamma) \mapsto (g\rho(\gamma),s(\gamma))$ is open.
\item \textit{closed} if the map $\gamma \mapsto ((r(\gamma),\rho(\gamma),s(\gamma))$ is closed.
\item \textit{faithful} if the map $\gamma \mapsto ((r(\gamma),\rho(\gamma),s(\gamma))$ is injective.
\end{enumerate}
We call a cocycle $\rho$ with all these properties a \textit{(T,C,F)-cocycle}.
\end{definition}

\begin{remark}
If $\Gamma$ is a discrete group then to prove that a cocycle $\rho$ is transverse it is enough to check if the map $\gamma \mapsto (\rho(\gamma), s(\gamma)$ is open \cite{MR1900993}.
\end{remark}

From such a cocycle we can construct a locally compact Hausdorff $\Gamma$-space $\Omega$ by considering the space $\G^{(0)}\times \Gamma$, equipped with the product topology. We define $\sim$ on $\G^{(0)}\times \Gamma$ by $(x,g)\sim (y,h)$ if there exists $\gamma \in \G$ with $s(\gamma)=x$, $r(\gamma)=y$ and $\rho(\gamma)=h^{-1}g$. Denote the quotient of $\G^{(0)}\times \Gamma$ by $\sim$, equipped with the quotient topology, by $\Omega$. The closed condition on the cocycle makes this space Hausdorff \cite{MR1900993}. When this construction is applied to a suitable partial action of a group, the space $\Omega$ is called the globalisation of the partial action \cite{MR2041539}. 

The main result of Khoskham and Skandalis \cite{MR1900993} described in the introduction is more precisely given again below:

\begin{theorem}\label{Thm:1.8}
Let $\rho: \G \rightarrow \Gamma$ be a continuous, faithful, closed, transverse cocycle. Then the space $\Omega$ defined above is a locally compact Hausdorff space and there is a Morita equivalence of $\G$ with $\Omega\rtimes \Gamma$.\qed
\end{theorem}

The important part of this construction was extended by Milan and Steinberg \cite{MR3231226} to cover the more general situation where the cocycles are inverse semigroup valued. Additionally, they show that every strongly 0-F-inverse monoid admits (T,C,F)-cocycle onto a group (see for example Corollary 6.17 \cite{MR3231226}). This result is what we will appeal to in later sections.

\subsection{The inverse monoid constructed from Example \ref{ex:1}}
Recall from Example \ref{ex:1} that to a subspace $X$ of $\Gamma$ it is possible to associate a collection $\mathcal{T}_{X}$ of partial translations of $X$. In this section we will understand the semigroup generated by $\mathcal{T}_{X}$.

\begin{lemma}\label{Lem:PTS}
Let $\Gamma$ be a finitely generated group and let $X \subset \Gamma$. Then inverse monoid, denoted $S$, generated by the collection of partial translations $\mathcal{T}_{X}$ is 0-F-inverse, with maximal element set $\mathcal{T}_{X}$. 
\end{lemma}
\begin{proof}
First we prove maximality of the translations. We prove that for any $s\in S \setminus \lbrace 0 \rbrace$ there exists a unique $t \in \mathcal{T}$ such that $s \leq t$. Observe that for each $g,h\in \Gamma$ the product $t_{g}t_{h}$ is contained, as a partial bijection, in $t_{gh}$. This implies that any product of elements of $\mathcal{T}$ is less than a unique $t \in \mathcal{T}$. As the action of $\Gamma$ on itself is transitive, the collection $\mathcal{T}$ partitions $X \times X$. From this we have that for any pair $t_{i},t_{j}\in \mathcal{T}$ $et_{i}=et_{j} \Leftrightarrow t_{i}=t_{j}$. 

Now we prove that $S$ is $0$-E-unitary. Let $e\in E(S)\setminus \lbrace 0 \rbrace$ and $s\in S\setminus \lbrace 0 \rbrace$. As any product of translations is contained in a unique translation, it is enough to consider the case that $s \in \mathcal{T}$. It follows that  the condition $e \leq s$ implies that $s$ fixes some elements of $X$. However, $\mathcal{T}$ partitions $X \times X$ and so $s \leq id_{X}$. We assumed that $s$ was maximal however, so $s = id_{X}$. Now any general word in $\mathcal{T}$ satisfies: $e \leq s \implies s \leq id_{X}$, hence $s$ is idempotent.  
\end{proof}

The construction from Example \ref{ex:1} is a special example of a \textit{partial action} of a group.

\begin{definition}
Let $\Gamma$ be a discrete group and let $X$ be a topological space. Then $\Gamma$ acts partially on $X$ by partial homeomorphisms if there is a map: $\theta: \Gamma \rightarrow I(X)$ that satisfies:
\begin{enumerate}
\item $\theta(e_{\Gamma})=Id_{X}$,
\item $\theta(g)\theta(h)\leq \theta(gh)$,
\item The domains and ranges of each $\theta(g)$ are open, and $\theta(g)$ is a homeomorphism between them.
\end{enumerate}
\end{definition}

Maps that satisfy condition (2) are called \textit{dual prehomomorphisms}, and such maps can be classified in terms of a universal group \cite{MR745358}:

\begin{definition}
Let $\Gamma$ be a discrete group. Consider the collection of pairs: $(X,g)$ for $\lbrace 1,g\rbrace \subset X$, where $X$ is a finite subset of $\Gamma$. The set of such $(X,g)$ is then equipped with a product and inverse:
\begin{equation*}
(X,g)(Y,h) = (X\cup gY,gh)\mbox{ , } (X,g)^{-1}=(g^{-1}X,g^{-1})
\end{equation*}
This inverse monoid is called the \textit{Birget-Rhodes (prefix) expansion of $\Gamma$}. This has maximal group homomorphic image $\Gamma$, and it has the universal property that it is the largest such inverse monoid. We denote this by $\Gamma^{Pr}$. The partial order on $\Gamma^{Pr}$ can be described by reverse inclusion, induced from reverse inclusion on finite subsets of $\Gamma$. It is F-inverse, with maximal elements: $\lbrace(\lbrace 1,g \rbrace, g):g \in \Gamma \rbrace$.
\end{definition}

This classifying monoid can be used to prove that the inverse semigroup generated by the collection $\mathcal{T}$ of Example \ref{ex:1} is strongly 0-F-inverse.

\begin{proposition}\label{Prop:3.37}
Let $S = \langle \theta_{g} | g \in \Gamma \rangle$, where $\theta: \Gamma \rightarrow S$ is a partial action. If $S$ is 0-F-inverse with $Max(S) = \lbrace \theta_{g} | g \in \Gamma \rbrace$. If for each $g \not e$ in $\Gamma$ such that $\theta_{g}$ is not zero we have $\theta_{g}$ is also not idempotent then $S$ is strongly 0-F-inverse.
\end{proposition}
\begin{proof}
We build a map $\Phi$ back onto $\Gamma^{0}$. Let $m: S\setminus \lbrace 0 \rbrace \rightarrow Max(S)$ be the map that sends each non-zero $s$ to the maximal element $m(s)$ above $s$ and consider the following diagram:
\begin{equation*}
\xymatrix{
\Gamma\ar@{->}[r]^{\theta}\ar@{->}[dr]^{}  & S\ar@{->}[dr]^{\Phi}  & \\
  & \Gamma^{pr} \ar@{->}[r]^{\sigma}\ar@{->}[u]^{\overline{\theta}}  & \Gamma^{0}
}
\end{equation*}
where $\Gamma^{pr}$ is the prefix expansion of $\Gamma$. Define the map $\Phi:S \rightarrow \Gamma^{0}$ by:
\begin{equation*}
\Phi(s)=\sigma ( m ( \overline{\theta}^{-1} (m(s)))), \Phi(0)=0
\end{equation*}
For each maximal element the preimage under $\overline{\theta}$ is well defined as the map $\theta_{g}$ has the property that $\theta_{g}=\theta_{h} \Rightarrow g=h$ precisely when $\theta_{g} \not = 0 \in S$. Given the preimage is a subset of the F-inverse monoid $\Gamma^{pr}$ we know that the maximal element in the preimage is the element $(\lbrace 1,g \rbrace,g)$ for each $g \in \Gamma$, from where we can conclude that the map $\sigma$ takes this onto $g \in \Gamma$.

We now prove it is a prehomomorphism. Let $\theta_{g},\theta_{h} \in S$, then:
\begin{eqnarray*}
\Phi(\theta_{g})=\sigma ( m(\overline{\theta}^{-1}(\theta_{g}))) = \sigma ( (\lbrace 1,g \rbrace, g) )= g\\
\Phi(\theta_{h})=\sigma ( m(\overline{\theta}^{-1}(\theta_{h}))) = \sigma ( (\lbrace 1,h \rbrace, h) )= h\\
\Phi(\theta_{gh})=\sigma ( m(\overline{\theta}^{-1}(\theta_{gh}))) = \sigma ( (\lbrace 1,gh \rbrace, gh) )= gh
\end{eqnarray*}
Hence whenever $\theta_{g},\theta_{h}$ and $\theta_{gh}$ are defined we know that $\Phi(\theta_{g}\theta_{h})=\Phi(\theta_{g})\Phi(\theta_{h})$. They fail to be defined if:
\begin{enumerate}
\item If $\theta_{gh} = 0$ in $S$ but $\theta_{g}$ and $\theta_{h} \not = 0$ in $S$, then $0=\Phi(\theta_{g}\theta_{h})\leq \Phi(\theta_{g})\Phi(\theta_{h})$

\item If (without loss of generality) $\theta_{g}=0$ then $0=\Phi(0.\theta_{h})= 0.\Phi(\theta_{h})=0$
\end{enumerate}
So prove that the inverse monoid $S$ is strongly 0-F-inverse it is enough to prove then that the map $\Phi$ is idempotent pure, and without loss of generality it is enough to consider maps of only the maximal elements - as the dual prehomomorphism property implies that in studying any word that is non-zero we will be less than some $\theta_{g}$ for some $g \in \Gamma$.

So consider the map $\Phi$ applied to a $\theta_{g}$:
\begin{equation*}
\Phi(\theta_{g})=\sigma ( m(\overline{\theta}^{-1}(\theta_{g}))) = \sigma ( (\lbrace 1,g \rbrace, g) )= g
\end{equation*}
Now assume that $\Phi(\theta_{g}) = e_{\Gamma}$. Then it follows that $\sigma (m (\overline{\theta}^{-1}(\theta_{g})))=e_{\Gamma}$. As $\sigma$ is idempotent pure, it follows then that $m(\overline{\theta}^{-1}(\theta_{g}))=1$, hence for any preimage $t\in \theta^{-1}(\theta_{g})$ we know that $t \leq 1$, and by the property of being 0-E-unitary it then follows that $t \in E(\Gamma^{pr})$. Mapping this back onto $\theta_{g}$ we can conclude that $\theta_{g}$ is idempotent, but by assumption this only occurs if $g = e$.\end{proof}

This has an immediate consequence for the inverse monoid generated by the construction from Example \ref{ex:1}.

\begin{corollary}\label{Cor:Trick}
Let $\Gamma$ be a finitely generated discrete group and let $X$ be coarsely embedded into $\Gamma$. Then the collection $\mathcal{T}_{X}$ arising from the construction in Example \ref{ex:1} generates a strongly $0$-F-inverse monoid. 
\end{corollary}
\begin{proof}
We remark that as $\Gamma$ acts on itself freely and so restricts to a free partial action of $\Gamma$ on $X$. The truncation has the property that no $t_{g}$ that is non-zero will be idempotent. The result now follows from Lemma \ref{Lem:PTS} and Proposition \ref{Prop:3.37}.
\end{proof}

\subsection{How these collections connect to $G(X)$}
Suppose that $X$ is a uniformly discrete metric space of bounded geometry satisfying the conditions of Corollary \ref{Cor:Trick}. Then from the collection $\mathcal{T}_{X}$ of partial translations that generate the metric coarse structure, it is possible to construct a second countable groupoid $G(\mathcal{T}_{X})$ that will decompose the coarse groupoid $G(X)$. Consider the inverse monoid $S$ generated by $\mathcal{T}_{X}$ and the representation, denoted $\pi_{X}$ in $I(X)$. We wish to take this representation into account when constructing the universal groupoid.

Consider the representation of the inverse monoid $S$ on $\ell^{2}(X)$ that naturally arises from $\pi_{X}$. We can complete the semigroup ring in this representation to get an algebra $C^{*}_{\pi_{X}}S$, which has a unital commutative subalgebra $C^{*}_{\pi_{X}}E$. By Proposition 10.6 \cite{MR2419901} we will get a closed subspace $\widehat{X}$ of unit space $\E$ of $\G_{\E}$ by considering the spectrum of $C^{*}_{\pi_{X}}E$. This will be invariant under the action of $S$, so we can reduce $\G_{\E}$ to this (see Section 10 of \cite{MR2419901} for all the details of this construction). We denote this groupoid by $G(\mathcal{T}_{X})$.

Theorem 10.16 \cite{MR2419901} implies that we have the isomorphism: $C^{*}_{r}(\G(\mathcal{T}_{X})) \cong C^{*}_{\pi_{X}}(S)\cong C^{*}\mathcal{T}_{X}$, where $C^{*}\mathcal{T}_{X}$ is the $C^{*}$-algebra generated by the maps $t \in \mathcal{T}_{X}$ that viewed as operators on $\ell^{2}X$ via $\pi_{X}$.

As the following claim will illustrate, the groupoid $G(\mathcal{T}_{X})$ is well controlled by the elements of $\mathcal{T}_{X}$.
\begin{claim}\label{Claim:C1}
Let $S$ be 0-F-inverse. Then every element $[s,\phi] \in \G_{\E}$ has a representative $[t,\phi]$ where $t$ is a maximal element.
\end{claim}
\begin{proof}
Take $t=t_{s}$ the unique maximal element above $s$. Then we know 
\begin{equation*}
s = t_{s}s^{*}s \mbox{ and } s^{*}s \leq t_{s}^{*}t_{s}
\end{equation*} 
The second equation tells us that $t_{s}^{*}t_{s} \in F_{\phi}$ as filters are upwardly closed, thus $(t_{s},\phi)$ is a valid element. Now to see $[t_{s},\phi]=[s,\phi]$ we need to find an $e \in E$ such that $e \in F_{\phi}$ and $se=t_{s}e$. Take $e=s^{*}s$ and then use the first equation to see that $s(s^{*}s)=t_{s}(s^{*}s)$.
\end{proof}

Thus, the translations are the only elements of $S$ that need to be concerned with when working with $\G(\mathcal{T}_{X})$. Finally, the groupoid $\G(\mathcal{T}_{X})$ acts freely on $\beta X$ and so we can generate now the coarse groupoid using this data using Lemma 3.3b) \cite{MR1905840}:

\begin{lemma}\label{Lem:CG}
The coarse groupoid $G(X) \cong \beta X \rtimes \G(\mathcal{T}_{X}).$\qed
\end{lemma}

\section{Applying these techniques to Monster groups}

In this section we connect the inverse semigroup picture developed in the previous section to Gromov monster groups.

\subsection{Final groupoid preliminaries}
\begin{definition}
Let $\G$ be an locally compact \'etale groupoid and let $F$ be a subset of the unit space $\G^{(0)}$. We say that $F$ is \textit{saturated} if for all $\gamma \in \G$ with $s(\gamma) \in F$ we also have $r(\gamma)\in F$. 
\end{definition}

Our aim is now to show that if we have a \'etale groupoid $\G$ that has a $(T,C,F)$ map to a group, we can get $(T,C,F)$ maps on reductions associated to certain saturated sets. We do this with the following two Lemmas.

\begin{lemma}\label{Lem:Cut}
Let $\G$ be an \'etale locally compact Hausdorff groupoid with a (T,C,F)-cocycle $\rho$ to $\Gamma$. Then relation $\sim$ on $\G^{(0)} \times \Gamma$ preserves saturated subsets of $\G^{(0)}$
\end{lemma}
\begin{proof}
Let $U$ be a saturated subset of $\G^{(0)}$ and let $x \in U$, $y \in U^{c}$. Assume for a contradiction that $(x,g) \sim (y,h)$ in $\G^{(0)} \times \Gamma$. Then there exists a $\gamma$ $\in \G$ such that $s(\gamma)=x$, $r(\gamma)=y$ and $\rho(\gamma)=g^{-1}h$, but as $U$ is saturated no such $\gamma$ exists. 
\end{proof}

\begin{lemma}\label{Lem:Top}
Let $\G$ be an \'etale locally compact Hausdorff groupoid and let $F$ be a closed saturated subset of $\G^{(0)}$. If $\G$ admits a (T,C,F)-cocycle $\rho$ onto a discrete group $\Gamma$ then so do $\G_{F}$ and $\G_{F^{c}}$. 
\end{lemma}
\begin{proof}
Observe that $\G_{F}$ is a closed subgroupoid and $\G_{F^{c}}$ is its open compliment. We consider them as topological groupoids in their own right using the subspace topology. We now check that these topologies are compatible with the subspace topologies in the appropriate places in the definition of $(T,C,F)$.
\begin{enumerate}
\item Faithful: as the map $P:\gamma \mapsto (r(\gamma),\rho(\gamma), s(\gamma))$ is injective, it is clear that its restriction to either $\G_{F^{c}}$ or $\G_{F}$ will also be injective.
\item Transverse: It is enough to show that $\lbrace (\rho(\gamma),s(\gamma)): \gamma \in \G_{F^{c}}\rbrace$ is open; this follows as it is precisely the intersection of $\lbrace (\rho(\gamma),s(\gamma)): \gamma \in \G \rbrace$ with $\Gamma \times F^{c}$. The same holds for $\G_{F}$.
\item Closed: We must show $P:\gamma \mapsto (r(\gamma),\rho(\gamma), s(\gamma))$ is closed. Let $V$ be a closed subset of $\G_{F^{c}}$. Then $V=V'\cap \G$, in the subspace topology. Now by saturation, we can conclude $P(V) = P(V^{'}) \cap F^{c}\times \Gamma \times F^{c}$, which is closed in the subspace topology coming from $ \G^{(0)}\times \Gamma \times \G^{(0)}$.
\end{enumerate} 
\end{proof}

Suppose the groupoid $\G(\mathcal{T}_{X})$ admits a (T,C,F)-cocycle onto a discrete group $\Gamma$. The following proposition outlines how to induce a (T,C,F)-cocycle on $G(X)$.

\begin{proposition}\label{Prop:Cocycle}
Let $X$ be a uniformly discrete metric space of bounded geometry and $\Gamma$ be a finitely generated group such that $X$ is coarsely embedded into $\Gamma$. Then the coarse groupoid $G(X)$ admits a (T,C,F)-cocycle onto $\Gamma$.
\end{proposition}
\begin{proof}
Immediately, we know that $\G:=\G(\mathcal{T})$ admits a (T,C,F)-cocycle $\rho$: this follows from the fact that the universal groupoid of the inverse monoid $S(\mathcal{T})$ does by Lemma \ref{Lem:Cut} and Corollary 6.17 of \cite{MR3231226}. Now we define a map $G(X) \rightarrow \Gamma$ using the anchor map $\pi: \beta X \twoheadrightarrow \X$:
\begin{eqnarray*}
\Phi: \beta X \rtimes \G \rightarrow \Gamma \\
(\omega, [t,\pi(\omega)]) \mapsto \rho(t)
\end{eqnarray*}
This map comes from observing that $X\rtimes \G \rightarrow \G^{(0)}\rtimes\G$ is a closed map (as $\pi$ is closed). As both $\G$ and $G(X)$ are principal it is easy to check that this induced map $\Phi$ is both closed and faithful. 

We now check that $\Phi$ is transitive. We check this on a basis for $G(X)$: let $U \subset G(X)$ be a slice. Then $(\Phi \times s)(U)=\Phi(U) \times s(U)$ is open as $\Gamma$ is discrete and $s(U)$ is a homeomorphism on $U$.
\end{proof}

Now we have two ingredients for the main results:
\begin{enumerate}
\item Local data arising from partial translations of $X$ induced from $\Gamma$ and the fact that from a combinatorial point of view these maps form an inverse monoid with rich structure and a groupoid decomposition of $G(X)$ that is induced by the fact these translations generate the metric in the appropriate sense.
\item Results that tell us that given a (T,C,F)-cocycle on $G(X)$ it is possible to get (T,C,F)-cocycles on $X\times X$ and $G(X)|_{\partial\beta X}$ that are compatible in a natural way.
\end{enumerate}

\subsection{Counterexamples to the Baum-Connes assembly conjecture}
Now we use coarsely embedded expanders in the constructions of the previous sections to the results indicated in the introduction.

\begin{definition}\label{Def:GMG}
A finitely generated discrete group $\Gamma$ is a \textit{Gromov monster group} if there exists a large girth expander with vertex degree uniformly bounded above $X$ and a coarse embedding $f: X \hookrightarrow \Gamma$. 
\end{definition}

This definition of Gromov monster group is quite strong: groups with weakly embedded expander graphs were first shown to exist by Gromov \cite{MR1978492}, with a detailed proof given by Arzhantseva, Delzant \cite{exrangrps}. A recent construction of Osajda \cite{Osajda.2014} simplifies and improves the construction: in fact, to get groups of the nature described in this definition relies completely on the work of Osajda. 

Using Proposition \ref{Prop:Cocycle} and Theorem \ref{Thm:1.8} we have the following:

\begin{theorem}\label{Thm:3.6}
Let $\Gamma$ be a Gromov monster group with coarsely embedded large girt expander $X$. Then there is a locally compact space $\Omega_{\beta X}$ such that $G(X)$ is Morita equivalent to $\Omega_{\beta X}\rtimes \Gamma$.\qed
\end{theorem}

It is well known for a large girth expander $X$ of uniformly bounded vertex degree that the Baum-Connes conjecture for $G(X)$ is injective, but not surjective \cite{mypub1,explg1}. This translates, via Theorem \ref{Thm:3.6}, to:

\begin{theorem}
Let $\Gamma$ be a Gromov monster group. Then the Baum-Connes conjecture for $\Gamma$ with coefficients in $C_{0}(Y_{\beta X})$ fails to be a surjection, but is an injection.\qed
\end{theorem}

\subsection{A different proof of non-K-exactness and the failure of the Baum-Connes conjecture}

We will now be much more explicit about this failure using the groupoid techniques from \cite{MR1911663}; we consider the associated ladder diagram in K-theory and K-homology coming from the decomposition of $G(X)$ into $G(X)|_{\partial\beta X}$ and $X\times X$. The following is essentially included in \cite{MR1911663}

\begin{theorem}Let $\Gamma$ be a finitely generated discrete group that coarsely contains a large girth expander $X$. Then $\Gamma$ is not $K$-exact.
\end{theorem}
\begin{proof}
Let $\mathcal{T}$ be the collection of partial translations obtained as in Example \ref{ex:1} by truncating the action of $\Gamma$ to $X$. Using Lemma \ref{Lem:CG} and Proposition \ref{Prop:Cocycle}, we can show that this collection decomposes the coarse groupoid $G(X)$ as $\beta X \rtimes \G(\mathcal{T})$, and subsequently that $G(X)$ admits a (T,C,F)-cocycle onto $\Gamma$. From this cocycle we construct a locally compact Hausdorff space $\Omega_{\beta X}$ using the relation $\sim$. It then follows from Theorem \ref{Thm:1.8} that $G(X) \cong \beta X \rtimes \G(\mathcal{T})$ is Morita equivalent to $\Omega_{\beta X}\rtimes \Gamma$. 

To construct the complete sequence we use Lemma \ref{Lem:Cut} to get $\Omega_{X}:= (X \times G)/\sim$ and $\Omega_{\partial\beta X}:= (\partial\beta X \times G)/\sim$. We then get the short exact sequence of $\Gamma$-algebras:
\begin{equation*}
0 \rightarrow C_{0}(\Omega_{X}) \rightarrow C_{0}(\Omega_{\beta X}) \rightarrow C_{0}(\Omega_{\partial \beta X}) \rightarrow 0.
\end{equation*}
We remark here also that by the theorem of Khoskham and Skandalis \cite{MR1900993} these Morita equivalences induce strong Morita equivalences of $C^{*}$-algebras. As $\Gamma$ is countable and $\beta X$ is $\sigma$-compact we can deduce that each of the enveloping spaces are also $\sigma$-compact, hence the cross product algebras are all $\sigma$-unital. From work of Rieffel \cite{MR679708} we have long exact sequences of K-theory groups in which all the vertical maps are isomorphisms:
\begin{equation*}
\xymatrix@=1em{...\ar[r] & K_{0}(C_{0}(\Omega_{X})\rtimes G) \ar[r]& K_{0}(C_{0}(\Omega_{\beta X})\rtimes G) \ar[r]& K_{0}(C_{0}(\Omega_{\partial\beta X})\rtimes G)\ar[r] & ...\\
...\ar[r] & K_{0}(\mathcal{K}) \ar[r]\ar[u]^{\ucong}& K_{0}(C^{*}_{r}(G(X))) \ar[r]\ar[u]^{\ucong}& K_{0}(C^{*}_{r}(G(X)|_{\partial\beta X})) \ar[r]\ar[u]^{\ucong}& ...}
\end{equation*}
We can conclude the result by observing that the bottom line is not exact as a sequence of abelian groups by either \cite{MR1911663}, \cite{explg1} or \cite{mypub1}. It follows therefore that the sequence:
\begin{equation*}
0 \rightarrow C_{0}(\Omega_{X})\rtimes_{r} G \rightarrow C_{0}(\Omega_{\beta X})\rtimes_{r} G \rightarrow C_{0}(\Omega_{\partial\beta X})\rtimes_{r} G \rightarrow 0
\end{equation*}
is not exact in the middle term.
\end{proof}

\section{Positive results for the Baum-Connes conjecture for a Gromov monster group}

The aim in this final section is to prove positive results about the Baum-Connes conjecture with certain coefficients for a Gromov monster group $\Gamma$. To do this we will extend the techniques in the previous section using ideas from \cite{mypub1}. To that end, we introduce the following definition: 

\begin{definition}
A uniformly discrete metric space $X$ with bounded geometry is said to be \textit{a-T-menable at infinity} if the coarse boundary groupoid $G(X)_{\partial\beta X}$ is a-T-menable in the sense of \cite{MR1703305}, i.e it admits a (locally) proper negative type function. 
\end{definition}

Examples of spaces that are a-T-menable at infinity are spaces that coarsely embed into Hilbert space, or more generally fibred coarsely embed into Hilbert space \cite{mypub2}. We recall the outcome of \cite{mypub1} in the following Proposition:

\begin{proposition}\label{Prop:Outsourced}
Let $X$ be a large girth expander with vertex degree uniformly bounded above. Then $X$ is a-T-menable at infinity. \qed
\end{proposition}

Using the Theorem \ref{Thm:1.8} and Proposition \ref{Prop:Outsourced} we will prove that the groupoid $Y_{\partial\beta X}\rtimes \Gamma$ is a-T-menable. From here, using results of Tu, we can conclude that the Baum-Connes conjecture holds for this groupoid with any coefficients.

\begin{theorem}\label{Thm:MT2}
Let $\Gamma$ be a finitely generated group that coarsely contains a large girth expander $X$ with uniformly bounded vertex degree. Then the groupoid $\Omega_{\partial\beta X}\rtimes \Gamma$ is a-T-menable.
\end{theorem}
\begin{proof}
By Proposition \ref{Prop:Cocycle} and Theorem \ref{Thm:1.8} the groupoid $G(X)|_{\partial\beta X}$ is Morita equivalent to $\Omega_{\partial\beta X}\rtimes \Gamma$. Explicitly,the decomposition of $G(X)$ as $\beta X \rtimes \G(\mathcal{T})$ for the natural translation structure associated to $X$ as a metric subspace of $\Gamma$ and Lemma \ref{Lem:Top} applied to the closed saturated subset $\partial\beta X$. The result now follows as a-T-menability for groupoids is an invariant of Morita equivalences (see \cite{MR1703305}).
\end{proof}

This has a natural corollary:

\begin{corollary}
Let $\Gamma$ be a finitely generated group that coarsely contains a large girth expander $X$. Then the Baum-Connes conjecture for $\Gamma$ with coefficients in any $(\Omega_{\partial\beta X}\rtimes \Gamma)$-$C^{*}$-algebra is an isomorphism.\qed
\end{corollary}

We remark that different techniques that rely on a-T-menability at infinity were considered in \cite{BGW-exact2013} to obtain a similar result.

\bibliography{ref.bib}

\begin{thebibliography}{BGW13}

\bibitem[AD08]{exrangrps}
G.N Arzhantseva and T~Delzant.
\newblock Examples of random groups.
\newblock {\em Journal of Topology (submitted)}, 2008.

\bibitem[BGW13]{BGW-exact2013}
Paul Baum, Erik Guentner, and Rufus Willett.
\newblock Expanders, exact crossed products, and the {B}aum-{C}onnes
  conjecture.
\newblock {\em http://arxiv.org/abs/1311.2343}, 2013.

\bibitem[BNW07]{MR2363428}
J.~Brodzki, G.~A. Niblo, and N.~J. Wright.
\newblock Property {A}, partial translation structures, and uniform embeddings
  in groups.
\newblock {\em J. Lond. Math. Soc. (2)}, 76(2):479--497, 2007.

\bibitem[BR84]{MR745358}
Jean-Camille Birget and John Rhodes.
\newblock Almost finite expansions of arbitrary semigroups.
\newblock {\em J. Pure Appl. Algebra}, 32(3):239--287, 1984.

\bibitem[Exe08]{MR2419901}
Ruy Exel.
\newblock Inverse semigroups and combinatorial {$C\sp \ast$}-algebras.
\newblock {\em Bull. Braz. Math. Soc. (N.S.)}, 39(2):191--313, 2008.

\bibitem[FS12]{mypub2}
Martin Finn-Sell.
\newblock Fibred coarse embedding into {H}ilbert space, a-{T}-menability and
  the coarse {N}ovikov conjecture.
\newblock {\em Preprint: http://arxiv.org/abs/1304.3348}, 2012.

\bibitem[FSW14]{mypub1}
Martin Finn-Sell and Nick Wright.
\newblock Spaces of graphs, boundary groupoids and the coarse {B}aum--{C}onnes
  conjecture.
\newblock {\em Adv. Math.}, 259:306--338, 2014.

\bibitem[Gro03]{MR1978492}
M.~Gromov.
\newblock Random walk in random groups.
\newblock {\em Geom. Funct. Anal.}, 13(1):73--146, 2003.

\bibitem[HLS02]{MR1911663}
N.~Higson, V.~Lafforgue, and G.~Skandalis.
\newblock Counterexamples to the {B}aum-{C}onnes conjecture.
\newblock {\em Geom. Funct. Anal.}, 12(2):330--354, 2002.

\bibitem[How95]{MR1455373}
John~M. Howie.
\newblock {\em Fundamentals of semigroup theory}, volume~12 of {\em London
  Mathematical Society Monographs. New Series}.
\newblock The Clarendon Press Oxford University Press, New York, 1995.
\newblock Oxford Science Publications.

\bibitem[KL04]{MR2041539}
J.~Kellendonk and Mark~V. Lawson.
\newblock Partial actions of groups.
\newblock {\em Internat. J. Algebra Comput.}, 14(1):87--114, 2004.

\bibitem[KS02]{MR1900993}
Mahmood Khoshkam and Georges Skandalis.
\newblock Regular representation of groupoid {$C^*$}-algebras and applications
  to inverse semigroups.
\newblock {\em J. Reine Angew. Math.}, 546:47--72, 2002.

\bibitem[Law12]{MR2974110}
M.~V. Lawson.
\newblock Non-commutative {S}tone duality: inverse semigroups, topological
  groupoids and {$C^\ast$}-algebras.
\newblock {\em Internat. J. Algebra Comput.}, 22(6):1250058, 47, 2012.

\bibitem[Len08]{MR2465914}
Daniel~H. Lenz.
\newblock On an order-based construction of a topological groupoid from an
  inverse semigroup.
\newblock {\em Proc. Edinb. Math. Soc. (2)}, 51(2):387--406, 2008.

\bibitem[LL13]{MR3077869}
Mark~V. Lawson and Daniel~H. Lenz.
\newblock Pseudogroups and their \'etale groupoids.
\newblock {\em Adv. Math.}, 244:117--170, 2013.

\bibitem[MS14]{MR3231226}
David Milan and Benjamin Steinberg.
\newblock On inverse semigroup {$C^*$}-algebras and crossed products.
\newblock {\em Groups Geom. Dyn.}, 8(2):485--512, 2014.

\bibitem[Osa14]{Osajda.2014}
Damian Osajda.
\newblock Small cancellation labellings of some infinite graphs and
  applications.
\newblock {\em arXiv:math/1406.5015}, 2014.

\bibitem[Rie82]{MR679708}
Marc~A. Rieffel.
\newblock Morita equivalence for operator algebras.
\newblock In {\em Operator algebras and applications, {P}art {I} ({K}ingston,
  {O}nt., 1980)}, volume~38 of {\em Proc. Sympos. Pure Math.}, pages 285--298.
  Amer. Math. Soc., Providence, R.I., 1982.

\bibitem[STY02]{MR1905840}
G.~Skandalis, J.~L. Tu, and G.~Yu.
\newblock The coarse {B}aum-{C}onnes conjecture and groupoids.
\newblock {\em Topology}, 41(4):807--834, 2002.

\bibitem[Tu99]{MR1703305}
Jean-Louis Tu.
\newblock La conjecture de {B}aum-{C}onnes pour les feuilletages moyennables.
\newblock {\em $K$-Theory}, 17(3):215--264, 1999.

\bibitem[WY12]{explg1}
Rufus Willett and Guoliang Yu.
\newblock Higher index theory for certain expanders and {G}romov monster
  groups, {I}.
\newblock {\em Adv. Math.}, 229(3):1380--1416, 2012.

\end{thebibliography}

\end{document}